\documentclass[10pt]{article}
\usepackage[T1]{fontenc}
\usepackage[cp1250]{inputenc}
\usepackage{lmodern}

\usepackage[english]{babel}
\usepackage{latexsym}
\usepackage{amsmath}
\usepackage{indentfirst}
\usepackage{amsthm}
\usepackage{amssymb}
\usepackage{amsfonts}
\usepackage{stackrel}
\usepackage{dsfont}
\usepackage{tikz}
\usepackage{slashbox}
\usepackage[mathscr]{eucal}
\usepackage{authblk}
\usepackage{hyperref}
\usepackage{mathabx}
\usepackage{bookmark}

\newtheorem{thm}{Theorem}
\newtheorem{lemma}[thm]{Lemma}

\newcommand{\reals}{\mathbb{R}}

\newcommand{\integers}{\mathbb{Z}}
\newcommand{\complex}{\mathbb{C}}

\newcommand{\Fourier}{\mathcal{F}}
\newcommand{\Laplace}{\mathcal{L}}
\newcommand{\cosine}{\mathcal{C}}

\newcommand{\COS}{\mathcal{COS}}

\begin{document}
\title{Integral transforms characterized by convolution}
\author{Mateusz Krukowski}
\affil{Institute of Mathematics, \L\'od\'z University of Technology, \\ W\'ol\-cza\'n\-ska 215, \
90-924 \ \L\'od\'z, \ Poland \\ \vspace{0.3cm} e-mail: mateusz.krukowski@p.lodz.pl}
\maketitle

\begin{abstract}
Inspired by Jaming's characterization of the Fourier transform on specific groups via the convolution property, we provide a novel approach which characterizes the Fourier transform on any locally compact abelian group. In particular, our characterization encompasses Jaming's results. Furthermore, we demonstrate that the cosine transform as well as the Laplace transform can also be characterized via a suitable convolution property.
\end{abstract}

\vspace{0.2cm}

\section{Introduction}
\label{Chapter:Introduction}

It is well-known that the Fourier transform satisfies the convolution property. More precisely, let $G$ be a locally compact abelian group with a Haar measure $\mu$ and a dual group $\widehat{G}.$ The Fourier transform is a map $\Fourier: L^1(G)\longrightarrow C_0(\widehat{G})$ given by the formula
\begin{gather}
\forall_{\substack{f\in L^1(G)\\ \chi\in \widehat{G}}}\ \Fourier(f)(\chi) := \int_G\ f(x)\overline{\chi(x)}\ d\mu(x),
\label{Fouriertransform}
\end{gather}

\noindent
whereas the Fourier convolution $\star_{\Fourier}: L^1(G)\times L^1(G)\longrightarrow L^1(G)$ is given by
\begin{gather}
\forall_{\substack{f,g\in L^1(G)\\ x\in G}}\ f\star_{\Fourier} g(x) := \int_G\ f(u)g(x-u)\ d\mu(u).
\label{Fourierconvolution}
\end{gather}

\noindent
For the sake of convenience, from this point onwards we will write $dx$ and $du$ instead of ``$d\mu(x)$'' and ``$d\mu(u)$'', respectively. The following Fourier convolution property (see Lemma 1.7.2 in \cite{DeitmarEchterhoff}, p. 30) holds:
$$\forall_{f,g\in L^1(G)}\ \Fourier(f\star_{\Fourier}g) = \Fourier(f)\Fourier(g).$$

Jaming proved that the convolution property characterizes (to a certain degree) the Fourier transform if $G = \reals, S^1, \integers$ or $\integers_n$ (see \cite{Jaming}). His paper inspired Lavanya and Thangavelu to show that any continuous *-homomorphism of $L^1(\complex^d)$ (with twisted convolution as multiplication) into $B(L^2(\reals^d))$ is essentially a Weyl transform and deduce a similar characterization for the group Fourier transform on the Heisenberg group (see \cite{LavanyaThangavelu} and \cite{LavanyaThangavelu2}). Furthermore, Kumar and Sivananthan went on to demonstrate that the convolution property characterizes the Fourier transform on compact groups (see \cite{KumarSivananthan}), while Alesker, Artstein-Avidan, Faifman and Milman studied the Fourier transform in terms of product preserving maps (see \cite{AleskerArtsteinAvidanFaifmanMilman}).

Studying the topic we realized two things that prompted us to write this paper. Firstly, Jaming's characterization of the Fourier transform need not be restricted to particular cases $G =  \reals, S^1, \integers$ or $\integers_n$. There is a unified approach for all locally compact abelian groups, which we demonstrate in the first part of Section \ref{Chapter:Fouriertransform}. 

Secondly, we discovered that the Fourier transform is not the only one that can be characterized via the convolution property. In the second part of Section \ref{Chapter:Fouriertransform} and in Section \ref{Chapter:Laplacetransform} we explain that cosine and Laplace convolutions characterize cosine and Laplace transforms, respectively.

\section{Fourier and cosine transform}
\label{Chapter:Fouriertransform}

As we have agreed in the Introduction, let $G$ stand for a locally compact abelian group with Haar measure $\mu$ and dual group $\widehat{G}.$ The formula for the Fourier transform and Fourier convolution are given by \eqref{Fouriertransform} and \ref{Fourierconvolution}, respectively.

\begin{lemma}
For every function $g\in L^1(G)$ and $x,y\in G$ the following equality holds
$$L_xg \star_{\Fourier} L_yg = g \star_{\Fourier} L_{x+y}g,$$

\noindent
where for every $z\in G$ the operator $L_z:L^1(G)\longrightarrow L^1(G)$ is given by 
\begin{gather*}
\forall_{u\in G}\ L_zf(u) := f(u - z).
\end{gather*}
\label{shiftoperators}
\end{lemma}
\begin{proof}
For every $u\in G$ we have 
\begin{equation*}
\begin{split}
L_xg\star_{\Fourier}L_yg(u) &= \int_G\ L_xg(v)L_yg(u-v)\ dv = \int_G\ g(v-x)g(u-v-y)\ dv\\
&\stackrel{v\mapsto v+x}{=} \int_G\ g(v)g(u-v-(x+y))\ dv = \int_G\ g(v)L_{x+y}g(u-v)\ dv = g\star_{\Fourier}L_{x+y}g(u),
\end{split}
\end{equation*}

\noindent
which completes the proof.
\end{proof}

The following theorem generalizes Theorems 2.1 and 3.1 in Jaming's paper (see \cite{Jaming}).

\begin{thm}
Let $T:L^1(G)\longrightarrow L^{\infty}(\widehat{G})$ be a linear and bounded operator. If it satisfies the Fourier convolution property
$$\forall_{f,g\in L^1(G)}\ T(f\star_{\Fourier} g) = T(f)T(g)$$ 

\noindent
then there exists a function $\theta_{\Fourier}: \widehat{G} \longrightarrow \widehat{G}\cup\{0\}$ such that 
\begin{gather}
\forall_{f\in L^1(G)}\ T(f) = \Fourier(f) \circ \theta_{\Fourier}
\label{mainresultFourier}
\end{gather}

\noindent
where $\theta_{\Fourier}(\phi) = 0$ if and only if $T(f)(\phi) = 0$ for every $f\in L^1(G).$
\label{characterizationofFouriertransform}
\end{thm}
\begin{proof}
We pick $\phi\in\widehat{G}.$ If $T(f)(\phi) = 0$ for every $f\in L^1(G)$, then we define $\theta_{\Fourier}(\phi) := 0$ so that the equality \eqref{mainresultFourier} holds. Suppose then that $\phi$ is such that $T(f)(\phi)\neq 0$ for some $f\in L^1(G)$. We consider a nonzero, linear functional $T_{\phi}: L^1(G)\longrightarrow \complex$ given by $T_{\phi}(f) := T(f)(\phi).$ We pick $g_* \in L^1(G)$ such that $T_{\phi}(g_*) = 1$ and define a function $\chi_{\phi} : G\longrightarrow \complex$ by the formula
$$\chi_{\phi}(x) := \overline{T_{\phi}(L_xg_*)}.$$

\noindent
Let us remark, that the choice of $g_*$ need not be unique, so there might be many functions $\chi_{\phi}$ corresponding to $\phi.$

We will now focus on proving various properties of the function $\chi_{\phi}.$ To begin with, we observe that 
\begin{equation*}
\forall_{x,y\in G}\ T_{\phi}(L_xg_*) T_{\phi}(L_yg_*) = T_{\phi}(L_xg_* \star_{\Fourier} L_yg_*) \stackrel{\text{Lemma}\ \ref{shiftoperators}}{=} T_{\phi}(g_* \star_{\Fourier} L_{x+y}g_*) = T_{\phi}(g_*)T_{\phi}(L_{x+y}g_*) = T_{\phi}(L_{x+y}g_*).
\end{equation*}

\noindent
Taking the complex conjugate reveals the equation
$$\forall_{x,y\in G}\ \chi_{\phi}(x)\chi_{\phi}(y) = \chi_{\phi}(x+y).$$

Furthermore, since $T_{\phi}$ is a continuous linear functional, then Lemma 1.4.2 in \cite{DeitmarEchterhoff}, p. 18 implies that $\chi_{\phi}$ is continuous. It is also nonzero (as $\chi_{\phi}(0) = 1$) and bounded, since 
\begin{gather}
\forall_{x\in G}\ |\chi_{\phi}(x)| \leqslant |T_{\phi}(L_xg_*)| \leqslant \|T_{\phi}\|\cdot \|L_xg_*\| = \|T_{\phi}\|\cdot \|g_*\|.
\label{boundednessofchi}
\end{gather} 

\noindent
Finally we argue that $|\chi_{\phi}(x)| = 1$ for every $x\in G.$ Indeed, suppose there exists $\bar{x} \in G$ such that $|\chi_{\phi}(\bar{x})| \neq 1.$ Since 
$$1 = |\chi_{\phi}(0)| = |\chi_{\phi}(\bar{x} - \bar{x})| = |\chi_{\phi}(\bar{x})| |\chi_{\phi}(-\bar{x})|$$

\noindent
then either $|\chi_{\phi}(\bar{x})| > 1$ or $|\chi_{\phi}(-\bar{x})| > 1.$ Without loss of generality, we may assume that the former is true. Consequently, we have 
$$\lim_{n\rightarrow \infty}\ |\chi_{\phi}(n\bar{x})| = |\chi_{\phi}(\bar{x})|^n = \infty,$$

\noindent
which contradicts boundedness of $\chi_{\phi}.$ Hence, we conclude that $|\chi_{\phi}(x)| = 1$ for every $x\in G.$ This means that $\chi_{\phi} \in \widehat{G}.$

Finally, using Lemma 11.45 in \cite{AliprantisBorder}, p. 427 (or Proposition 7 in \cite{Dinculeanu}, p. 123) we compute
\begin{equation*}
\begin{split}
\int_G\ f(x)\overline{\chi_{\phi}(x)}\ dx &= \int_G\ f(x) T_{\phi}(L_xg_*)\ dx = T_{\phi}\left(\int_G\ f(x) L_xg_*\ dx\right) \\
&= T_{\phi}(f\star_{\Fourier} g_*) = T_{\phi}(f) T_{\phi}(g_*) = T_{\phi}(f), 
\end{split}
\end{equation*}

\noindent
which concludes the proof if we put $\theta_{\Fourier}(\phi) := \chi_{\phi}.$
\end{proof}

We move on to show that the cosine transform admits a similar convolution characterization. The cosine transform is a map $\cosine : L^1(G)\longrightarrow C_0(\widehat{G})$ given by the formula 
\begin{gather}
\forall_{\substack{f,g\in L^1(G)\\ \chi\in \COS(G)}}\ \cosine(f)(\chi) := \int_G\ f(x)\chi(x)\ dx,
\label{cosinetransform}
\end{gather}

\noindent
where 
\begin{gather}
\COS(G) := \bigg\{\chi \in C^b(G)\ :\ \chi \neq 0,\ \forall_{x,y\in G}\ \chi(x)\chi(y) = \frac{\chi(x+y) + \chi(x-y)}{2}\bigg\}.
\label{cosineclass}
\end{gather}

\noindent
Functions in $\COS(G)$ are called cosine functions on group $G$, since for $G = \reals$ we have 
$$\COS(G) = \bigg\{x\mapsto \cos(yx)\ :\ y\in\reals\bigg\}.$$

\noindent
The cosine convolution $\star_{\cosine}:L^1(G)\times L^1(G)\longrightarrow  L^1(G)$ is given by
\begin{gather}
\forall_{x\in G}\ f\star_{\cosine} g(x) := \int_G\ f(u)\cdot \frac{g(x+u) + g(x-u)}{2}\ du.
\label{cosineconvolution}
\end{gather}

Let us recall a property of the cosine convolution which will be crucial in the sequel (see Theorem 2 in \cite{Krukowski}):

\begin{lemma}
Let $x,y\in G.$ If $g\in L^1(G)$ is an even function, then 
\begin{gather}
L_yg \star_{\cosine} L_xg = g\star_{\cosine} \frac{L_{x+y}g + L_{x-y}g}{2}.
\label{LxgstarcLyg}
\end{gather}
\label{dAlembertproperty}
\end{lemma}

Before we present a characterization of the cosine transform we need one more technical result:

\begin{lemma}
Let $T:L^1(G)\longrightarrow L^{\infty}(\COS(G))$ be a linear and bounded operator, which satisfies the cosine convolution property
\begin{gather}
\forall_{f,g\in L^1(G)}\ T(f\star_{\cosine} g) = T(f)T(g).
\label{cosineconvolutionproperty}
\end{gather}

\noindent
If $\phi \in \COS(G)$ is such that $T(f_*)(\phi) \neq 0$ for some $f_*\in L^1(G)$, then there exists an even function $g_*\in L^1(G)$ such that $T(g_*)(\phi)\neq 0.$
\label{existenceofevenfunction}
\end{lemma}
\begin{proof}
Let $\iota:G\longrightarrow G$ be the inverse function $\iota(x) := -x$ and let $T_{\phi}: L^1(G)\longrightarrow \complex$ be a nonzero, linear functional given by $T_{\phi}(f) := T(f)(\phi).$  Then
\begin{equation*}
\begin{split}
\forall_{f\in L^1(G)}\ (f\circ\iota) \star_{\cosine} f(x) &= \int_G\ f\circ\iota(u) \cdot \frac{f(x+u) + f(x-u)}{2}\ du\\
&\stackrel{u\mapsto -u}{=} \int_G\ f(u) \cdot \frac{f(x-u) + f(x+u)}{2}\ du = f\star_{\cosine} f,
\end{split}
\end{equation*}

\noindent
which leads to 
\begin{gather*}
\forall_{f\in L^1(G)}\ T_{\phi}(f\circ\iota)T_{\phi}(f) = T_{\phi}((f\circ\iota) \star_{\cosine} f) = T_{\phi}(f\star_{\cosine} f) = T_{\phi}(f)^2.
\label{canwedividebymf}
\end{gather*}

\noindent
Consequently, we have $T_{\phi}(f_*\circ\iota) = T_{\phi}(f_*).$ Finally, we put $g_* := f_* + f_*\circ\iota$ and observe that 
$$T_{\phi}(g_*) = T_{\phi}(f_*+f_*\circ\iota) = T_{\phi}(f_*) + T_{\phi}(f_*\circ\iota) = 2T_{\phi}(f_*) \neq 0,$$

\noindent
which concludes the proof.  
\end{proof}

We are now ready to demonstrate a counterpart of Theorem \ref{characterizationofFouriertransform} for the cosine transform:

\begin{thm}
Let $T:L^1(G)\longrightarrow L^{\infty}(\COS(G))$ be a linear and bounded operator. If it satisfies the cosine convolution property \eqref{cosineconvolutionproperty}, then there exists a function $\theta_{\cosine}: \COS(G) \longrightarrow \COS(G)\cup\{0\}$ such that 
\begin{gather}
\forall_{f\in L^1(G)}\ T(f) = \cosine(f) \circ \theta_{\cosine}
\label{mainresultcosine}
\end{gather}

\noindent
where $\theta_{\cosine}(\phi) = 0$ if and only if $T(f)(\phi) = 0$ for every $f\in L^1(G).$
\end{thm}
\begin{proof}
The proof follows the same lines as the proof of Theorem \ref{characterizationofFouriertransform}. We pick $\phi\in\COS(G)$ and check if $T(f)(\phi) = 0$ for every $f\in L^1(G)$. If so, we define $\theta_{\cosine}(\phi) := 0.$ Otherwise, if $T(f)(\phi) \neq 0$ for some $f\in L^1(G),$ then we consider a nonzero, linear functional $T_{\phi}: L^1(G)\longrightarrow \complex$ given by $T_{\phi}(f) := T(f)(\phi).$ By Lemma \ref{existenceofevenfunction} there exists an even function $g_* \in L^1(G)$ such that $T_{\phi}(g_*) = 1.$ We may thus define a function $\chi_{\phi} : G\longrightarrow \complex$ by the formula
$$\chi_{\phi}(x) := T_{\phi}(L_xg_*).$$

As in Theorem \ref{characterizationofFouriertransform} we study the properties of $\chi_{\phi}$ and see that 
\begin{itemize}
	\item it is nonzero due to $\chi_{\phi}(0) =1,$
	\item it is bounded due to \eqref{boundednessofchi},
	\item it is continuous due to Lemma 1.4.2 in \cite{DeitmarEchterhoff}, p. 18.
\end{itemize}

\noindent
Furthermore, we have 
\begin{equation*}
\begin{split}
\forall_{x,y\in G}\ T_{\phi}(L_xg_*) T_{\phi}(L_yg_*) &= T_{\phi}(L_xg_* \star_c L_yg_*) \stackrel{\text{Lemma}\ \ref{dAlembertproperty}}{=} T_{\phi}\left(g_* \star_{\cosine} \frac{L_{x+y}g_* + L_{x-y}g_*}{2}\right) \\
&= T_{\phi}(g_*) \cdot \frac{T_{\phi}(L_{x+y}g_*) + T_{\phi}(L_{x-y}g_*)}{2}= \frac{T_{\phi}(L_{x+y}g_*) + T_{\phi}(L_{x-y}g_*)}{2},
\end{split}
\end{equation*}

\noindent
which can be written as 
$$\forall_{x,y\in G}\ \chi_{\phi}(x)\chi_{\phi}(y) = \frac{\chi_{\phi}(x+y) + \chi_{\phi}(x-y)}{2}.$$

\noindent
We conclude that $\chi_{\phi} \in \COS(G),$ which in particular means that 
\begin{gather}
\forall_{y\in G}\ \chi_{\phi}(y) = \frac{\chi_{\phi}(y) + \chi_{\phi}(-y)}{2}.
\label{chiasmean}
\end{gather}

Finally, using Lemma 11.45 in \cite{AliprantisBorder}, p. 427 (or Proposition 7 in \cite{Dinculeanu}, p. 123) we compute
\begin{equation*}
\begin{split}
\int_G\ f(x) \chi_{\phi}(x)\ dx &\stackrel{\eqref{chiasmean}}{=} \int_G\ f(x) \cdot \frac{\chi_{\phi}(x) + \chi_{\phi}(-x)}{2}\ dx = \int_G\ f(x) \cdot \frac{T_{\phi}(L_xg_*) + T_{\phi}(L_{-x}g_*)}{2}\ dx \\
&= T_{\phi}\left(\int_G\ f(x) \cdot \frac{L_xg_* + L_{-x}g_*}{2}\ dx\right) = T_{\phi}(f\star_{\cosine} g_*) = T_{\phi}(f) T_{\phi}(g_*) = T_{\phi}(f), 
\end{split}
\end{equation*}

\noindent
which concludes the proof if we put $\theta_{\cosine}(\phi) := \chi_{\phi}.$
\end{proof}

\section{Laplace transform}
\label{Chapter:Laplacetransform}

In the previous section we focused on Fourier and cosine transforms and characterized them via suitable convolution properties. The purpose of the current section is to demonstrate that the Laplace transform enjoys a similar characterization. Let us recall that the Laplace transform is a map 
$\Laplace:L^1(\reals_+)\longrightarrow C_0(\reals_+)$ given by
$$\forall_{\substack{f\in L^1(\reals_+)\\ y\in\reals_+}}\ \Laplace(f)(y) := \int_0^{\infty}\ e^{-yx}f(x)\ dx,$$

\noindent
whereas the Laplace convolution $\star_{\Laplace}:L^1(\reals_+)\times L^1(\reals_+)\longrightarrow L^1(\reals_+)$ is given by
\begin{gather}
\forall_{\substack{f,g\in L^1(\reals_+)\\ x \in \reals_+}}\ f\star_{\Laplace} g(x) := \int_0^x\ f(u)g(x-u)\ du.
\label{Laplaceconvolution}
\end{gather}

\noindent
It is well-known that the Laplace transform satisfies the equality (see Theorem 2.39 in \cite{Schiff}, p. 92):
\begin{gather}
\forall_{f,g\in L^1(\reals_+)}\ \Laplace(f\star_{\Laplace}g) = \Laplace(f)\Laplace(g).
\label{Laplaceconv}
\end{gather}

\noindent
Similarly to the previous section, our goal is to ``reverse the implication'' and ask whether there are any other operators satisfying \eqref{Laplaceconv}. A (almost) negative answer will constitute a characterization of the Laplace transform in terms of the Laplace convolution \eqref{Laplaceconvolution}. To this end, we need the following auxilary lemma:

\begin{lemma}
Let $(\Omega,\Sigma,\mu)$ be a measure space, $g\in L^{\infty}(\Omega)$ and let $D$ be a dense subset in $L^1(\Omega)$. If 
$$\forall_{f\in D}\ \int_{\Omega}\ f(x)g(x)\ dx = 0$$

\noindent
then $g = 0.$
\label{duBoisReymondlemma}
\end{lemma}
\begin{proof}
We define a linear and bounded functional $T: L^1(\Omega) \longrightarrow \complex$ with the formula
$$T(f) := \int_{\Omega}\ f(x)g(x)\ dx.$$

\noindent
Since $D$ is dense in $L^1(\Omega)$ and $T|_D = 0$ then by Theorem 1.7 in \cite{ReedSimon}, p. 9 we conclude that $T(f) = 0$ for every $f\in L^1(\Omega).$ In particular, $T(g) = 0,$ which implies that $g = 0.$   
\end{proof}

At last, we are ready to characterize the Laplace transform in terms of the Laplace convolution property:

\begin{thm}
Let $T:L^1(\reals_+)\longrightarrow L^{\infty}(\reals_+)$ be a linear and bounded operator. If it satisfies the Laplace convolution property
\begin{gather}
\forall_{f,g\in L^1(G)}\ T(f\star_{\Laplace} g) = T(f)T(g)
\label{Laplaceconvolutionproperty}
\end{gather}

\noindent
then there exists a function $\theta_{\Laplace}: \reals_+ \longrightarrow \complex$ such that 
$$\forall_{f\in L^1(\reals_+)}\ T(f) = \Laplace(f) \circ \theta_{\Laplace}$$

\noindent
where $\theta_{\Laplace}(y) = 0$ if and only if $T(f)(y) = 0$ for every $f\in L^1(\reals_+).$
\end{thm}
\begin{proof}
We pick $y\in\reals_+$ and check whether $T(f)(y) = 0$ for every $f\in L^1(\reals_+).$ If so, we define $\theta_{\Laplace}(y) := 0.$ Otherwise, if $T(f)(y) \neq 0$ for some $f\in L^1(\reals_+),$ then we consider a nonzero, linear functional $T_y : L^1(\reals_+)\longrightarrow \complex$ given by $T_y(f) := T(f)(y).$ By Theorem 5.2.16 in \cite{Bobrowski}, p. 161 there exists a function $\chi_y\in L^{\infty}(\reals_+)$ such that 
$$\forall_{f\in L^1(G)}\ T_y(f) = \int_0^{\infty}\ f(x)\chi_y(x)\ dx.$$

By Fubini's theorem (see Theorem B.3.1 in \cite{DeitmarEchterhoff}, p. 287 or Theorem 8.8 in \cite{Rudin}, p. 164) we have
\begin{equation*}
\begin{split}
\forall_{f,g\in L^1(\reals_+)}\ T_y(f\star_{\Laplace}g) &= \int_0^{\infty}\ f\star_{\Laplace}g(u)\chi_y(u)\ du = \int_0^{\infty} \int_0^u\ f(v)g(u-v)\chi_y(u)\ dvdu \\
&= \int_0^{\infty} \int_v^{\infty}\ f(v)g(u-v)\chi_y(u)\ dudv \stackrel{u \mapsto u + v}{=} \int_0^{\infty}\int_0^{\infty}\ f(v)g(u)\chi_y(u+v)\ dudv.
\end{split}
\end{equation*}

\noindent
Since $T_y$ is nonzero we pick $g_*\in C_c(\reals_+)$ such that $T_y(g_*) = 1.$ By the Laplace convolution property \eqref{Laplaceconvolutionproperty} we have 
$$\forall_{f\in L^1(\reals_+)}\ T_y(f\star_{\Laplace}g_*) = T_y(f)T_y(g_*) = T_y(f),$$

\noindent
so
$$\forall_{f\in L^1(\reals_+)}\ \int_0^{\infty}\int_0^{\infty}\ f(v)g_*(u)\chi_y(u+v)\ dudv = \int_0^{\infty}\ f(v)\chi_y(v)\ dv,$$

\noindent
which we rewrite as
$$\forall_{f\in L^1(\reals_+)}\ \int_0^{\infty} f(v) \left(\int_0^{\infty}\ g_*(u)\chi_y(u+v)\ du - \chi_y(v)\right) dv = 0.$$

\noindent
Since $v\mapsto \int_0^{\infty}\ g_*(u)\chi_y(u+v)\ du - \chi_y(v)$ is a $L^{\infty}-$function then by Lemma \ref{duBoisReymondlemma} we conclude that 
$$\forall_{v\in\reals_+}\ \chi_y(v) = \int_0^{\infty}\ g_*(u)\chi_y(u+v)\ du.$$

\noindent
We use this integral functional equation to establish continuity of $\chi_y.$ For a fixed $v_*\in\reals_+$ and $h\in (-v_*,v_*)$ we have
\begin{equation*}
\begin{split}
|\chi_y(v_* + h) - \chi_y(v_*)| &= \left| \int_0^{\infty}\ g_*(u)\chi_y(u+v_*+h)\ du - \int_0^{\infty}\ g_*(u)\chi_y(u+v_*)\ du\right| \\
&=  \left| \int_{v_*+h}^{\infty}\ g_*(u-v_*-h)\chi_y(u)\ du - \int_{v_*}^{\infty}\ g_*(u-v_*)\chi_y(u)\ du\right| \\
&\leqslant \int_{v_*+h}^{\infty}\ |g_*(u-v_*-h) - g_*(u-v_*)| |\chi_y(u)|\ du + \int_{v_*}^{v_*+h}\ |g_*(u-v_*)| |\chi_y(u)|\ du \\
&\leqslant \|\chi_y\|_{\infty} \left(\int_{v_*+h}^{\infty}\ |g_*(u-v_*-h) - g_*(u-v_*)|\ du + \|g_*\|_{\infty} h\right).
\end{split}
\end{equation*}

\noindent
Since 
$$\lim_{h\rightarrow 0}\ \int_{v_*+h}^{\infty}\ |g_*(u-v_*-h) - g_*(u-v_*)|\ du = 0$$

\noindent
due to the dominated convergence theorem, then 
$$\lim_{h\rightarrow 0} |\chi_y(v_* + h) - \chi_y(v_*)| = 0,$$

\noindent
establishing the continuity of $\chi_y$.

For the second time we use the Laplace convolution property \eqref{Laplaceconvolutionproperty} to obtain 
$$\forall_{f,g\in L^1(\reals_+)}\ \int_0^{\infty}\int_0^{\infty}\ f(v)g(u)\chi_y(u+v)\ dudv = \int_0^{\infty}\ \int_0^{\infty}\ f(v)g(u)\chi_y(v)\chi_y(u)\ dudv,$$

\noindent
which we rewrite as 
$$\forall_{f,g\in L^1(\reals_+)}\ \int_0^{\infty}\int_0^{\infty}\ f(v)g(u)\bigg(\chi_y(u+v) - \chi_y(v)\chi_y(u)\bigg)\ dudv = 0.$$

\noindent
Since $(u,v)\mapsto \chi_y(u+v) - \chi_y(v)\chi_y(u)$ is a $L^{\infty}-$function, then by Lemma \ref{duBoisReymondlemma} we have
$$\forall_{u,v\in\reals_+}\ \chi_y(u+v) = \chi_y(v)\chi_y(u).$$

\noindent
This functional equation proves that if $\chi_y(\bar{u}) = 0$ for some $\bar{u}\in\reals_+$ then $\chi_y = 0$ and consequently, $T_y = 0.$ We have already dealt with this case at the beginning of the proof, so we assume that $\chi_y(u) \neq 0$ for every $u\in\reals_+.$ By Theorem 1.6.11 in \cite{Bobrowski}, p. 36 (and boundedness of $\chi_y$) we obtain $\chi_y(u) = e^{-zu}$ for some $z\in\reals_+$ (which is dependent on $y$). This concludes the proof if we put $\theta_{\Laplace}(y) := z.$
\end{proof}

\end{document}